\documentclass{article}
\usepackage[utf8]{inputenc}
\usepackage{enumerate}
\usepackage{amssymb, bm}
\usepackage{amsthm}
\usepackage{amsmath}
\usepackage{bbm}
 \usepackage[mathscr]{euscript}
 \usepackage[all]{xy}
\title{On a vanishing theorem due to Bogomolov}
\author{Xiaojun WU}
\date{\today}
\newtheorem{mythm}{Theorem}
\newtheorem{mylem}{Lemma}
\newtheorem{myprop}{Proposition}

\newtheorem{mydef}{Definition}
\newtheorem{myrem}{Remark}
\begin{document}
\def\cI{\mathcal{I}}
\def\Z{\mathbb{Z}}
\def\Q{\mathbb{Q}}  \def\C{\mathbb{C}}
 \def\R{\mathbb{R}}
 \def\N{\mathbb{N}}
 \def\H{\mathbb{H}}
  \def\P{\mathbb{P}}
 \def\rC{\mathrm{C}}
  \def\nd{\mathrm{nd}}
  \def\d{\partial}
 \def\dbar{{\overline{\partial}}}
\def\dzbar{{\overline{dz}}}
 \def\ii{\mathrm{i}}
  \def\d{\partial}
 \def\dbar{{\overline{\partial}}}
\def\dzbar{{\overline{dz}}}
\def \ddbar {\partial \overline{\partial}}
\def\cD{\mathbb{D}}
\def\cE{\mathcal{E}}  \def\cO{\mathcal{O}}
\def\P{\mathbb{P}}
\def\cI{\mathcal{I}}
\def \loc{\mathrm{loc}}
\def \log{\mathrm{log}}
\def \Im{\mathfrak{Im}}
\def \cC{\mathcal{C}}
\bibliographystyle{plain}
\def \dim{\mathrm{dim}}
\def \RHS{\mathrm{RHS}}
\def \liminf{\mathrm{liminf}}
\def \ker{\mathrm{Ker}}
\maketitle
\begin{abstract}
In this note, we give a new proof of a vanishing result originally due to Bogomolov, and later generalised by Mourougane and Boucksom. The statement holds for arbitrary pseudoeffective line bundles over compact K\"ahler manifolds, under an assumption on the numerical dimension of the line bundle.
\end{abstract}

\section{Introduction}
Let $L$ be a holomorphic line bundle over a compact K\"ahler manifold $X$. One well known vanishing theorem due to Bogomolov \cite{Bog} (at least in the projective algebraic case) asserts that
$$H^0(X, \Omega^p_X \otimes L^{-1})=0$$
for $p < \kappa(L)$.
In \cite{Mou}, the following two versions of the Bogomolov vanishing theorem are stated.
\begin{mythm}
If $L$ is a nef line bundle over a compact K\"ahler manifold $X$, then
$$H^0(X, \Omega^p_X \otimes L^{-1})=0$$
for $p < \nd(L)$. 
\end{mythm}
\begin{mythm}
If $L$ is a psef line bundle over a compact K\"ahler manifold $X$, one defines
$e(L)$ to be the largest natural number $k$ such that there exists a positive $(1,1)$-current $T \in c_1(L)$ whose absolutely continuous part has rank $k$ on a strictly positive Lebesgue measure set on $X$. Then $$H^0(X, \Omega^p_X \otimes L^{-1})=0$$
for $p < e(L)$.
\end{mythm}

The absolutely continuous part of a positive current referred to here is
taken in the sense of the Lebesgue decomposition theorem for measures; as
is weel known, such a decomposition is unique.
In this note, we give the following improved version of the Bogomolov vanishing theorem, following ideas of \cite{Mou}.
\begin{mythm}
Let $L$ be a psef line bundle over a compact Kähler manifold $X$. Then we have
$$H^0(X, \Omega^p_X \otimes L^{-1})=0$$
if $p < \nd(L)$.
\end{mythm}
The concept of numerical dimension $\nd(L)$ is recalled below.
We should mention that the above version of the Bogomolov vanishing theorem was first proven in \cite{Bou02}.
The strategy of the both proofs is based on the nef case proven in \cite{Mou}.
The difficulty is the control by Monge-Amp\`ere equation in the pseudo-effective case.
The difficulty is overcome in \cite{Bou02} by a singular version of Monge-Amp\`ere equation, and we give here another proof that only requires solving ``classical'' Monge-Amp\`ere equations. (For example, we avoid the use of Bonavero's version of holomorphic Morse inequalities cf. \cite{Dem85a}, \cite{Dem85b}, \cite{Bon93}.)

\section{Numerical dimension}
We first recall the K\"ahler version of the definition of numerical dimension, as stated e.g.\ in \cite{Dem14}. For $L$ a psef line bundle on a compact K\"ahler manifold $(X, \omega)$, we define
$$\nd(L): = \max\{p \in [0,n] ;\exists c >0,\forall \varepsilon >0, \exists h_{\varepsilon},i\Theta_{L,h_{\varepsilon}} \geq -\varepsilon \omega ,\mathrm{with}\,
\int_{X\setminus Z_{\varepsilon}}( i \Theta_{L,h_{\varepsilon}}+\varepsilon \omega)^p \wedge \omega^{n-p} \geq c\}.$$
Here the metrics $h_{\varepsilon}$ are supposed to have analytic singularities and $Z_{\varepsilon}$ is the singular set of the metric.
Fix a family of metric $h_{\varepsilon}$ as stated in the definition.
For such metrics, for $p > \nd(L)$, by definition,
$$\lim_{\varepsilon \to 0} \int_{X\setminus Z_{\varepsilon}}( i \Theta_{L,h_{\varepsilon}}+\varepsilon \omega)^p \wedge \omega^{n-p}=0.$$
If the line bundle $L$ is nef,  we can take $h_{\varepsilon}$ to be smooth and $Z_{\varepsilon}= \emptyset$,
(cf. the proof of point (i) in \cite{BDPP}, or \cite{Bou02}) and we have for any $p$
$$\lim_{\varepsilon \to 0} \int_{X\setminus Z_{\varepsilon}}( i \Theta_{L,h_{\varepsilon}}+\varepsilon \omega)^p \wedge \omega^{n-p}=\lim_{\varepsilon \to 0} \int_X (c_1(L)+\varepsilon \omega)^p \wedge \omega^{n-p}=\int_{X}  c_1(L)^p \wedge \omega^{n-p}.$$
The integral condition in the definition of the numerical dimension in the nef case means that  $p=\nd(L)$ is the largest integer such that
$$\int_{X}  c_1(L)^p \wedge \omega^{n-p} \neq 0.$$
Since for each $p$, $c_1(L)^p$ can be represented by a positive closed $(p,p)$-current, the triviality of the mass is equivalent to the triviality of the current.
In other words,
$$\nd(L)=\max\{p; c_1(L)^p  \neq 0\},$$
which corresponds to the definition of the numerical dimension for a nef line bundle.

In fact, if ones denotes $\alpha:= c_1(L)$, the numerical dimension of the psef line bundle $L$ is equal to the numerical dimension of the class of $\alpha$ defined in \cite{BEGZ}. To check the quality, one needs the definition of the moving intersection product for arbitrary psef $(1,1)$-classes $\alpha$ and any number of factors $1 \leq p \leq n$. We start by recalling the following definition.
\begin{mydef} {\rm (See \cite{DPS01})}.
Let $\varphi_1,\varphi_2$ be two quasi-psh functions on $X$ (i.e. $i \d \dbar \varphi_i \geq -C \omega$ in the sense of currents for some $C \geq 0$). Then, $\varphi_1$ is less singular than $\varphi_2$ (and write $\varphi_1 \preceq \varphi_2$) if we have $\varphi_2 \leq \varphi_1+C_1$ for some constant $C_1$.
Let $\alpha$ be a psef class in $H^{1,1}_{BC}(X,\R)$ and $\gamma$ be a smooth real $(1,1)$-form.
Let $T_1,T_2,\theta \in \alpha$
with $\theta$ smooth and such that $T_i=\theta+i\d \dbar \varphi_i (i=1,2)$.
$\varphi_i$ is well defined up to constant since $X$ is compact.
We say $T_1 \preceq T_2$ if and only if $\varphi_1 \preceq \varphi_2$.

The minimal element $T_{min,\gamma}$ with the pre-order relation $\preceq$ exists by taking the upper semi-continuous envelope of all $\varphi_{i}$ such that $\theta +\gamma+i \d \dbar \varphi_i \geq 0$ and $\sup_X \varphi_i=0$.
\end{mydef}
The positive product defined in \cite{BEGZ} is the real $(p,p)$ cohomology class $\langle \alpha^p \rangle$ of the limit 
$$\langle \alpha^p \rangle:=\lim_{\delta \to 0} \{\langle T_{\min, \delta \omega}^p\rangle\}$$
where $T_{\min, \delta \omega}$ is the positive current with minimal singularity in the class $\alpha + \delta \{\omega\}$ and $\langle T_{\min, \delta \omega}^p \rangle$ is the non-pluripolar product.
The numerical dimension of $\alpha$ is defined by
$$\nd(\alpha):= \max \{p| \langle \alpha^p \rangle \neq 0\}$$
which is also equal to $\max \{p|\int_X \langle \alpha^p \rangle  \wedge \omega^{n-p} >0\}$.
The equivalence of the two numerical dimensions given here is an adapted version of the arguments of \cite{Tos}.
Although this is perhaps well-known to experts, we feel useful to give complete proofs here. We need the definition of the non-K\"ahler locus given in \cite{Bou04}.
\begin{mydef}
Let $\alpha$ be a big class in $H^{1,1}(X,\R)$. The non-K\"ahler locus is defined to be
$$E_{nK}(\alpha):= \bigcap_{T \in \alpha} E_+(T)$$
where $T$ ranges all K\"ahler currents in $\alpha$ and $E_+(T):= \bigcup_{c >0} E_c(T)$. 
\end{mydef}
We will also need the following lemma of \cite{Bou04}, which implies in particular that the non-K\"ahler locus is in fact an analytic set. 
\begin{mylem}
Let $\alpha$ be a big class. There exists a K\"ahler current $\tilde{T}$ with analytic singularities such that $E_{nK}(\alpha)= E_+ (\tilde{T})$.
\end{mylem}
\begin{proof}
By means of a regularization, we have equivalently that
$E_{nK} (\alpha)=\bigcap_{T \in \alpha}E_+(T)$
where $T$ ranges all K\"ahler currents with analytic singularities.
Since $T$ has analytic singularities, $E_+(T)$ is a proper analytic set.
By the strong Noetherian property, there exist finitely many K\"ahler
currents $T_i (i \in I)$ with analytic singularities such that
$E_{nK}(T)= \bigcap_{i \in I} E_+(T_i)$.
Take a regularization $\tilde{T}$ of $\min_{i \in I} T_i$ (associated with
the max of potentials). Then we have
$$\nu(\tilde{T},x) \leq \min_{i \in I} \nu(T_i,x)$$
for any $x \in X$. In particular, this implies that
$$E_+(\tilde{T}) \subset \bigcap_{i \in I} E_+(T_i).$$
Since $\tilde{T}$ itself is a K\"ahler current with analytic singularities, we get in fact an equality in the statement.
\end{proof}
We will need the following result stated in \cite[Prop.\ 1.16]{BEGZ}.
\begin{myprop}
For $j=1,\cdots,p$, let $T_j$ and $T'_j$ be two closed positive $(1,1)$-currents with small unbounded locus (i.e. there exists a (locally) complete pluripolar closed subset $A$ of $X$ outside which the potential is locally bounded) in the same cohomology class, and assume also that $T_j$ is less singular than $T'_j$. Then the cohomology classes of their non-pluripolar products satisfy
$\{\langle T_1 \wedge \cdots \wedge T_p\rangle\} \geq \{\langle T'_1 \wedge \cdots \wedge T'_p \rangle\}$
in $H^{p,p}(X,\R)$, where $\geq $ means that the difference is pseudo-effective, i.e. representable by a closed positive $(p,p)$-current.
\end{myprop}
Now we are prepared to prove that
\begin{myprop}
For $L$ a psef line bundle, we have
$$\nd(c_1(L))=\nd(L).$$
\end{myprop}
\begin{proof}
Let $h_{\varepsilon}$ be a family of metric with analytic singularities as stated in the definition of $\nd(L)$.
Denote $A_{\varepsilon}:=E_{nK}(\alpha+\varepsilon \{\omega\})$.
Since $T_{\min, \varepsilon \omega} \preceq i \Theta_{L, h_{\varepsilon}}+\varepsilon \omega $, we have by proposition 1 that for any $1 \leq p \leq n$
$$\int_{X \setminus Z_{\varepsilon}} (i \Theta_{L, h_{\varepsilon}}+\varepsilon \omega)^p \wedge \omega^{n-p} = \int_{X\setminus (Z_{\varepsilon} \cup A_{\varepsilon})} \langle (i \Theta_{L, h_{\varepsilon}}+\varepsilon \omega)^p \rangle \wedge \omega^{n-p} \leq \int_{X\setminus (Z_{\varepsilon} \cup A_{\varepsilon})} \langle T_{\min, \varepsilon \omega}^p \rangle \wedge \omega^{n-p}.$$
The right hand term is the same as $\int_{X} \langle T_{\min, \varepsilon \omega}^p \rangle \wedge \omega^{n-p}$, 
since the non-pluripolar product has no mass on any analytic set.
It has limit equal to $\int_X \langle c_1(L)^p \rangle \wedge \omega^{n-p}$.
In particular, this implies that $\nd(c_1(L))\geq \nd(L)$.
We remark that $A_{\varepsilon}$ is an analytic set hence is a small unbounded locus.

For the other direction, we construct a family of metrics with analytic singularities with control of the Monge-Amp\`ere mass from below. Denote $p:= \nd(c_1(L))$.
Since $\langle c_1(L)^p \rangle \neq 0$, for $\varepsilon$ small enough, $\langle T_{\min, \varepsilon \omega}^p \rangle \neq 0$ and
$$\int_X \langle T_{\min, \varepsilon \omega}^p \rangle \wedge \omega^{n-p} \geq c$$
for some uniform constant $c >0$ whenever $\varepsilon$ is small enough.
Let $T_{\varepsilon,\delta}$ be a sequence of regularisation of $T_{\min, \varepsilon \omega}$ with analytic singularities such that 
$$T_{\varepsilon, \delta} \geq -\delta \omega$$
and the potentials of $T_{\varepsilon, \delta}$ decrease to the potential of $T_{\min, \varepsilon \omega}$.
Then $T_{\min, \varepsilon \omega}+\varepsilon \omega$ and $T_{\varepsilon , \delta}+\varepsilon \omega$ are closed positive currents in the cohomology class $\alpha + 2\varepsilon \{\omega\}$ if $\delta \leq \varepsilon$.
By lemma 1, $A_{\varepsilon}=E_+ (T_{\varepsilon})$ for some K\"ahler current with analytic singularities.
Thus $T_{\min, \varepsilon \omega} \preceq T_{\varepsilon}$, whose potential is locally bounded outside $A_{\varepsilon}$, as the potential of $T_{\varepsilon}$ is.
So the potentials of $T_{\varepsilon, \delta}$ are also locally bounded outside $A_{\varepsilon}$.
By weak continuity of the Bedford-Taylor Monge-Amp\`ere operators with respect to decreasing sequences of functions, we have on $X \setminus A_{\varepsilon}$ that, 
$$(T_{\varepsilon , \delta}+\varepsilon \omega)^l \to (T_{\min, \varepsilon \omega}+\varepsilon \omega)^l$$
for any $l$. 
By the Fatou lemma, we have
$$\int_{X \setminus A_{\varepsilon}} (T_{\min, \varepsilon \omega}+\varepsilon \omega)^p \wedge \omega^{n-p} \leq \liminf_{\delta \to 0}\int_{X \setminus A_{\varepsilon}} (T_{ \varepsilon , \delta}+\varepsilon \omega)^p \wedge \omega^{n-p}.$$
Take any family $\delta(\varepsilon)$ such that $\delta(\varepsilon) \leq \varepsilon$ and $\lim_{\varepsilon \to 0} \delta(\varepsilon) =0$.
Let $h_{\varepsilon}$ be the metric on $L$ with analytic singularities such that $i \Theta_{L, h_{\varepsilon}}=T_{ \varepsilon , \delta(\varepsilon)}-\varepsilon \omega$.
The metric $h_{\varepsilon}$ is uniquely defined up to a multiple.
To normalise it, we can assume for example, that the maximum of the potentials on $X$ equals to 0.
Hence we have
$$i \Theta_{L,h_{\varepsilon}}\geq -(\varepsilon+\delta(\varepsilon)) \omega \geq -2 \varepsilon \omega.$$
Then the sequence of metric satisfies the condition demanded in the definition of $\nd(L)$.
\end{proof}
\begin{myrem}{\rm
$ T_{\min, \varepsilon' \omega} \preceq T_{\min, \varepsilon \omega}+(\varepsilon'-\varepsilon)\omega $ for any $\varepsilon \leq \varepsilon'$.
Denote $T_{\min, \varepsilon \omega}= \theta+\varepsilon \omega +i \d \dbar \varphi_{\min, \varepsilon \omega}$.
We can arrange that 
$$\varphi_{\min,0} \leq \varphi_{\min,\varepsilon \omega} \leq \varphi_{\min,\varepsilon' \omega}.$$
The Bergman kernel regularisation perserves the ordering of potentials (cf. \cite{Dem14}), 
so we have
$$\varphi_{0, \delta} \leq \varphi_{\varepsilon, \delta} \leq \varphi_{\varepsilon',\delta}.$$ for any $\delta >0$.
If $\delta(\varepsilon)$ is increasing with respect to $\varepsilon$, by the proof of the proposition, 
we can choose the metric $h_{\varepsilon}$ to be decreasing with respect to $\varepsilon$.
The limit of $\varphi_{\varepsilon, \delta(\varepsilon)}$ as $\varepsilon \to 0$ is equal to $\varphi_{\min,0}$ corresponding to the metric with minimal singularities on $L$.}
\end{myrem}
\begin{myrem}
{\rm By similarity with the definition of \cite{Dem14} for the numerical dimension of a psef line bundle, one can define the numerical dimension of a psef cohomology class.
The above proof shows in fact that the two definitions of the numerical dimension of a psef cohomology class coincide. }
\end{myrem}
\section{Bogomolov vanishing theorem}
In this section, we prove a numerical dimension version of the Bogomolov vanishing theorem.
From now on, we denote $l := \nd(L)$.
Then we have
$$\lambda_{\varepsilon}:= \frac{\int_{X\setminus Z_{\varepsilon}}( i \Theta_{L,h_{\varepsilon}}+\varepsilon \omega)^n}{\int_X \omega^n} \geq c \varepsilon^{n-l}.$$

The first step of the proof consists in the use of Yau's theorem \cite{Yau78}, so as to show that one can turn the above integral inequality into a pointwise lower bound, more precisely, the inequality $(*)$ given below.
Up to a re-parametrisation of $\varepsilon$, we can assume that
$$i \Theta_{L,h_{\varepsilon}}+\varepsilon \omega \geq \frac{\varepsilon}{2}\omega.$$
Let $\nu_{\varepsilon}: X_{\varepsilon} \to X$ be a log resolution of the analytic singularities of $h_\varepsilon$.
We then have
$$\nu_{\varepsilon}^* (i \Theta_{L,h_{\varepsilon}}+\varepsilon \omega)=  [D_{\varepsilon}]+\beta_{\varepsilon}$$
where $\beta_{\varepsilon} \geq \frac{\varepsilon}{2} \nu_{\varepsilon}^*\omega \geq 0$ is a smooth positive closed $(1,1)$-form on $X_{\varepsilon}$.
It is strictly positive on the complement $X_{\varepsilon} \setminus E$ of the exceptional divisor $E$ (we denote its irreducible components as $E_l$).
$[D_{\varepsilon}]$ is the closed positive current associated to a $\R$-divisor.
By the theorem of Hironaka \cite{Hir64} we can assume that the exceptional divisor is simple normal crossing divisors and the morphism is obtained as a composition of a sequence of blow up with smooth centres. 
In this situation, there exist arbitrary small numbers $\eta_l >0$ such that the cohomological class of $\beta_{\varepsilon} -\sum \eta_l [E_l]$ is a K\"ahler class (which means that there exists a K\"ahler form in this class).

Hence we can find a quasi psh function $\hat{\theta_{\varepsilon}}$ on $X_{\varepsilon}$ such that
$$\hat{\beta_{\varepsilon}}:=\beta_{\varepsilon} -\sum \eta_l [E_l]+i \d\dbar \hat{\theta_{\varepsilon}} $$
is a K\"ahler metric on $X_{\varepsilon}$.
By taking $\eta_l$ small enough, we can assume that
$$\int_{X_{\varepsilon}} (\hat{\beta_{\varepsilon}})^n \geq \frac{1}{2}\int_X \beta_{\varepsilon}^n.$$
 
The assumption on the numerical dimension implies that there exists $c>0$ such that with $Z_{\varepsilon}:=\nu_{\varepsilon}(E) \subset X $, we have
$$\int_{X_{\varepsilon}} \beta_{\varepsilon}^n =\int_{X \setminus Z_{\varepsilon}} (i \Theta_{L,h_{\varepsilon}}+\varepsilon \omega)^n$$
$$\geq {n\choose l} (\frac{\varepsilon}{2})^{n-l} \int_{X \setminus Z_{\varepsilon}}  (i \Theta_{L,h_{\varepsilon}}+\frac{\varepsilon}{2} \omega)^l \wedge \omega^{n-l} \geq c \varepsilon^{n-l} \int_X \omega^n.$$
Hence we have
$$\int_{X_{\varepsilon}}  (\hat{\beta_{\varepsilon}})^n \geq \frac{c}{2} \varepsilon^{n-l} \int_X \omega^n.$$ 
By Yau's theorem \cite{Yau78}, there exists a quasi-psh potential $\hat{\tau_{\varepsilon}}$ on $X_{\varepsilon}$ 
such that $\hat{\beta}_{\varepsilon}+i \d \dbar \hat{\tau}_{\varepsilon}$
is a K\"ahler metric on $X_{\varepsilon}$ with any prescribed volume form $\hat{f}$ such that $\int_{X_{\varepsilon}} \hat{f}= \int_{X_{\varepsilon}}(\hat{\beta_{\varepsilon}})^n $.
By the integral condition, we can choose a smooth volume form on $X_{\varepsilon}$ such that 
$$\hat{f} > \frac{c}{3} \varepsilon^{n-l} \nu_{\varepsilon}^* \omega^n\leqno(*)$$ 
everywhere on $X_{\varepsilon}$.
Fix $h$ a smooth metric on $L$ and let $\varphi_{\varepsilon}$ be the weight function of $h_{\varepsilon}$ (i.e. $h_{\varepsilon}=h e^{-2 \varphi_{\varepsilon}}$).
We impose the additional normalization condition that  $\sup_{X_{\varepsilon}}(\nu_{\varepsilon}^*\varphi_{\varepsilon}+\hat{\theta_{\varepsilon}}+
\hat{\tau_{\varepsilon}})=0$.

We now work again on $X$ (e.g. by taking direct images to construct a sequence of singular metrics on $X$). Consider $\theta_{\varepsilon}:= \nu_{\varepsilon*} \hat{\theta}_{\varepsilon}$ and
$\tau_{\varepsilon}:=  \nu_{\varepsilon*} \hat{\tau}_{\varepsilon} \in L^1_{\loc}(X)$.
Define $\Phi_{\varepsilon}:= \varphi_{\varepsilon}+\theta_{\varepsilon}+\tau_{\varepsilon}$. This is a quasi psh potential on $X$ since
it satisfies the condition
$$\nu_{\varepsilon}^*(i \Theta_{L,h}+\varepsilon \omega +i \d\dbar \Phi_{\varepsilon})
=[D_{\varepsilon}]+\sum \eta_l [E_l]+\hat{\beta_{\varepsilon}}+i \d\dbar \hat{\tau_{\varepsilon}} \geq 0.$$
Define $\tilde{Z_{\varepsilon}}:= \nu_{\varepsilon}(D_{\varepsilon})$ which includes $Z_{\varepsilon}$ since the support of the divisor $D_{\varepsilon}$ includes all components of the exceptional divisor by Hironaka theorem \cite{Hir64}.
By construction, $\Phi_{\varepsilon}$ is smooth on $X \setminus \tilde{Z_{\varepsilon}}$.
By the normalisation condition we have $\sup_X \Phi_{\varepsilon}=0$.
Since $i \Theta_{L,h} +\varepsilon \omega+i \d \dbar \Phi_{\varepsilon}$ is a family of $(1,1)$-forms in a bounded family of cohomology classes, with the above normalisation, we have,  up to taking a subsequence, that the family of quasi-psh potentials $\Phi_{\varepsilon}$ converges almost everywhere to $\Phi \in L^1(X)$ by weak compactness.
It satisfies that 
$$i \Theta_{L,h}+ i \d \dbar \Phi \geq 0.$$

We also have that 
$$\nu_{\varepsilon}^* \mathbb{1}_{X \setminus \tilde{Z_{\varepsilon}}}(i \Theta_{L,h}+i \d\dbar \Phi_{\varepsilon} + \varepsilon \omega)^n \geq \hat{\beta_{\varepsilon}}^n \geq \frac{c}{3} \varepsilon^{n-l} \nu_{\varepsilon}^* \omega^n.$$
In other words, on $X \setminus \tilde{Z_{\varepsilon}}$
$$(i \Theta_{L,h}+i \d\dbar \Phi_{\varepsilon} + \varepsilon \omega)^n \geq \frac{c}{3} \varepsilon^{n-l} \omega^n. $$
 
To use the Bochner-Kodaira-Nakano inequality, we need to change the K\"ahler metric in such a way that $X\setminus \tilde{Z}_{\varepsilon}$ becomes a complete manifold. 
We define a family of K\"ahler metrics $\omega_{\varepsilon,\delta}:=\omega+\delta (i\d\dbar \psi_{\varepsilon}+\omega)$, 
for   $\delta >0$ which is complete metrics on $X \setminus \tilde{Z}_{\varepsilon}$,
 where $\psi_{\varepsilon}$ is a quasi-psh function on $X$ with $\psi_{\varepsilon}=-\infty$ on $\tilde{Z}_{\varepsilon}$, $\psi_{\varepsilon}$ smooth on $X \setminus \tilde{Z}_{\varepsilon}$ and   $i\ddbar \psi_{\varepsilon}+\omega \geq 0$ (see e.g. \cite{Dem82}, Th\'eor\`eme 1.5).
 
Here we choose $\psi_{\varepsilon}$ more explicit for better control. Since we will use the Bochner-Kodaira-Nakano inequality on $X\setminus \tilde{Z}_{\varepsilon}$, to simplify the notations, we identify it with $X_{\varepsilon} \setminus \mathrm{Supp} (D_{\varepsilon})$.
We define 
$$\psi_{\varepsilon}:=-\sqrt {- \nu_{\varepsilon}^* \varphi_{\varepsilon}-C}$$ with $ C \in \R$ such that 
$\sup_{X_{\varepsilon}}  \nu_{\varepsilon}^* \varphi_{\varepsilon}+C=-1$.
Now $\psi_{\varepsilon}$ satisfies the condition of \cite{Dem82}, Th\'eor\`eme 1.5 following its calculation.

We want to prove that $e^{\Phi_{\varepsilon}}dV_{\omega_{\varepsilon, \delta}}$ is a current on $X_{\varepsilon}$. 
Since $X_{\varepsilon}$ is compact, it has finite mass on $X_{\varepsilon}$.
In particular, it has finite mass on $X_{\varepsilon} \setminus \mathrm{Supp} (D_{\varepsilon})$.
It is enough to prove that $e^{\Phi_{\varepsilon}} (i \d \dbar \psi_{\varepsilon})^p$ defines a current on $X_{\varepsilon}$ for any $p >0$.
More precisely, we prove that $e^{ \nu_{\varepsilon}^* \varphi_{\varepsilon}} (i \d \dbar \psi_{\varepsilon})^p$ defines a current on $X_{\varepsilon}$ for any $p >0$.
Since 
$$i \d \dbar \psi_{\varepsilon}=\frac{ i \d \dbar \nu_{\varepsilon}^* \varphi}{2\psi_{\varepsilon}}+\frac{ i \d \nu_{\varepsilon}^* \varphi \wedge \dbar \nu_{\varepsilon}^* \varphi}{4 \psi_{\varepsilon}},$$
it is equivalent to prove that 
$e^{\nu_{\varepsilon}^* \varphi} (i \d \dbar \nu_{\varepsilon}^* \varphi)^p \wedge (i \d \nu_{\varepsilon}^* \varphi \wedge \dbar \nu_{\varepsilon}^* \varphi)^q$
($p,q \geq 0$)
defines a current.
By anti-commutativity, we can assume $q$ is either 0 or 1.
$$e^{\frac{\nu_{\varepsilon}^* \varphi_{\varepsilon}}{n}} i \d \dbar \nu_{\varepsilon}^*  \varphi_{\varepsilon}= e^{\frac{\nu_{\varepsilon}^*  \varphi_{\varepsilon}}{n}} ([D_{\varepsilon}]+\mathrm{smooth} \; \mathrm{terms})=e^{\frac{\nu_{\varepsilon}^*  \varphi_{\varepsilon}}{n}} \mathrm{smooth} \; \mathrm{terms}$$
since  $\nu_{\varepsilon}^*  \varphi_{\varepsilon}$ vanishes along $D_{\varepsilon}$.
Thus it is smooth on $X_{\varepsilon}$ vanishing along $D_{\varepsilon}$.

On the other hand, in local coordinates, 
$$\nu_{\varepsilon}^*  \varphi_{\varepsilon}=\sum \alpha_i \log(|z_i|^2)$$
with $\alpha_i >0$.
So 
$$e^{\frac{\nu_{\varepsilon}^* \varphi_{\varepsilon}}{n}} i \d \nu_{\varepsilon}^* \varphi_{\varepsilon} \wedge \dbar \nu_{\varepsilon}^*  \varphi_{\varepsilon}=\prod |z_k|^{\frac{2 \alpha_k}{n}} i \sum \frac{\alpha_i dz_i}{z_i} \wedge \sum \overline{\frac{\alpha_j dz_j}{z_j}}$$
has all coefficients in $L^1_{\loc}$. 
(This is because $\alpha_k >0$ for any $k$, although a priori, the derivative of a quasi-psh function is not necessarily in $L^2_{\loc}$.)
Hence the current is well defined as a wedge product of locally integrable functions and smooth forms.

In conclusion, we have $\int_{X\setminus Z_{\varepsilon}} e^{\Phi_{\varepsilon}} dV_{\omega_{\varepsilon,\delta}}$ is finite and uniformly bounded for $\delta$ small enough.

\begin{myrem}
{\rm
Let us indicate an alternative argument in a more general situation, following
a suggestion by Demailly. It is not necessary for our proof, but may be
interesting for other uses.
Let $(X, \omega)$ be a compact Hermitian manifold and $D$ a SNC divisor in $X$.
Let $u \in H^0(X,\mathcal{C}^0_{p,q,X} \otimes L)$ be a $(p,q)$ continuous forms with value in some line bundle $(L,h)$ endowed with some continuous metric $h$.
In this remark, we construct a family of complete metrics $\omega_{\delta}$ on $X \setminus D$ such that $\omega_{\delta}$ decreasing to $\omega$ as $\delta \to 0$ and
$$\int_{X\setminus D} |u|^2_h dV_{\omega_{\delta}} \leq C$$
where $C$ is a universal constant independent of $\delta$.

To begin with, we recall some facts about the local model: the Poincar\'e metric on the punctured disk.
The Poincar\'e metric on $\H:=\{z \in \C| \Im(z)>0 \}$ is given by $\frac{i dz \wedge \overline{dz}}{|\Im z|^2}$.
There exists an infinite cover from $\mathcal{H}$ to $\mathcal{D}^*=\{z \in \C| |z|<1\}$ given by $z \mapsto e^{iz}$.
The Poincar\'e metric on $\mathcal{H}$ is the pull back of the Poincar\'e metric on $\mathcal{D}^*$ given by
$$\frac{i dz \wedge \overline{dz}}{|z|^2\,|\log(|z|)|^2}.$$
Since the Poincar\'e metric on $\mathcal{H}$ is complete and the cover is locally diffeomorphism, the Poincar\'e metric on $\mathcal{D}^*$ is also geodesic complete.
It is well known that the Poincar\'e metric is of volume finite near the origin:
$$\int_{0<|z|<\frac{1}{2}}\frac{i dz \wedge \overline{dz}}{|z|^2||\log(|z|)|^2}=\int_0^{2 \pi} d\theta \int_{0}^{\frac{1}{2}} \frac{dr}{r (\log(r))^2}=\frac{2 \pi}{\log 2} < \infty.$$
Now we return to the construction of our metrics.
Let $U_{\alpha}$ be a finite system of coordinate charts of $X$($X$ is compact) such that for any $U_{\alpha}$ such that
$U_{\alpha} \cap D \neq \emptyset$, (we denote the set of all such indices as $I$) we have in this coordinate chart
$$U_{\alpha} \cap D=\{z_1=\cdots=z_r=0\},$$
$$U_{\alpha} \subset \{ |z_i|<1, \forall i\}.$$
This is possible since $D$ is a SNC divisor.
Let $\chi_{\alpha}$ be a partition of unity adapted to this cover.
Define  the family of metric $\omega_{\delta}$ on $X \setminus D$ as follows:
$$g_{\alpha}:=\sum_{i=1}^r \frac{i dz_i \wedge \overline{dz_i}}{|z_i|^2\,|\log(|z_i|)|^2}+\sum_{i=r+1}^n i dz_i \wedge \overline{dz_i}$$
$$\omega_{\delta}:= \omega +\delta \sum_{\alpha \in I} \chi_{\alpha} g_{\alpha}.$$
The sum converges since we take finite sums.
We have by construction $\omega_{\delta} \geq \omega$ decreasing to $\omega$.
By a similar calculation to the one made above, we have 
$$\int_{X\setminus D} |u|^2_h dV_{\omega_{\delta}} \leq C.$$
We remark that $|u|^2_h$ a priori depends on $\omega_\delta$.
However in a local chart $U_{\alpha}(\alpha \in I)$, we can write
$$u=\sum_{J,K, |J|=p, |K|=q} u_{J,K} dz^J \wedge \overline{dz^K}$$ using $J$(resp. $K$) for multi index of length $p$ (resp. $q$).
Denote for $1 \leq i \leq n$ and a multi index $I$, $\delta_{i I}=1$ if $i \in I$ and $\delta_{i I}=0$ if $i \notin I$.
Then we have
$$|u|^2_{g_{\alpha},h}=\sum_{J,K} \prod_{i=1}^n (|z_i|\;|\log |z_i||)^{\delta_{iJ}} \prod_{i=1}^n (|z_i|\;|\log |z_i||)^{\delta_{iK}} |u_{JK}|^2 e^{-2 \varphi_{\alpha}}$$
where $\varphi_{\alpha}$ is the weight function of $h$ on $U_{\alpha}$ (i.e. $h=e^{-2\varphi_{\alpha}}$ on $U_{\alpha}$).
Hence $|u|^2_{g_{\alpha},h}$ is uniformly bounded since
$|z_i||\log |z_i|| $ is bounded for $|z_i|<1$ and all terms are continuous.

It remains to prove that $\omega_{\delta}$ is complete.
We prove it by contradiction.
Let $\gamma(t)$ be a geodesic of $\omega_{\delta}$ with natural parametrization for $\delta >0$ whose maximal defining interval is $]t_0,t_1[$ with $t_1 <\infty$.
By property of ordinary differential equation
(the solution goes outside any compact subset), the adherent point(s) must be contained in $D$ with respect to the background topology of $X$. Since $X$ is compact, there exists a sequence $\gamma(t_{\nu}) \to x\in D$ with $t_{\nu} \to t_1$, $t_{\nu} <t_1$. Up to taking a subsequence we can assume that such a sequence is contained in some chart $U_{\alpha_0}$.
Then $|\gamma'(t)|_{g_{\alpha_0}} \leq \delta|\gamma'(t)|_{\omega_{\delta}}=\delta$
where the second equality is from the fact that $\gamma(t)$ be a geodesic of $\omega_{\delta}$. Hence $\gamma(t_{\nu})$ is a Cauchy sequence with respect to $g_{\alpha}$.
Since the Poincar\'e type metric $g_{\alpha}$ is complete, the limit $x \in U_{\alpha_0} \setminus D$ exists, which gives a contradiction.
}
\end{myrem}

We recall the Bochner-Kodaira-Nakano inequality in the non compact case.
\begin{mythm}
Let $h$ be a smooth hermitian metric on $L$ over $(X, \omega)$ a complete K\"ahler manifold.
We assume that  the  curvature  possesses  a  uniform  lower  bound
$$i \Theta_{L,h} \geq -C \omega.$$
Then for an arbitrary $(p,q)$-form $u \in C^{\infty}(X,\wedge^{p,q}T^{*}_X \otimes L)$  which is $L^2$ integrable, the following basic a priori inequality holds
$$\parallel \dbar u \parallel^2+\parallel \dbar^* u \parallel^ 2 \geq \int_X\langle [i \Theta_{L,h},\Lambda]u,u \rangle dV_{\omega}.$$
\end{mythm}
\begin{proof}
For $u$ with compact support, the inequality is just the classical one. When $u$ is just $L^2$-integrable case, since $(X,\omega)$ is assumed to be complete, there exists a sequence of smooth forms $u_{\nu} $ with compact support in $X$ (obtained for example by truncating $u$ and taking the convolution with a regularizing kernel)  such  that $u_{\nu} \to u$ in $L^2$ and such  that $\dbar u_{\nu} \to \dbar u,\dbar^* u_{\nu} \to \dbar^* u$ in $L^2$.  

By our curvature assumption the term on the right is controlled by $C |u|^2$ which is $L^2$.
We thus get the inequality by passing to the limit, using Lebesgue’s dominated convergence theorem.
\end{proof}

We now return to the proof of the Bogomolov vanishing theorem.

Let $u$ be a holomorphic p-form with value in $L^{-1}$. We take the metric induced from $(L, h e^{- \Phi_{\varepsilon}})$.
The Bochner-Kodaira-Nakano inequality on the complete manifold $(X \setminus \tilde{Z}_{\varepsilon}, \omega_{\varepsilon,\delta})$ gives
$$0 \geq \int_{X\setminus Z_{\varepsilon} }\langle [i \Theta_{L,h},\Lambda]u,u \rangle e^{\Phi_{\varepsilon}} dV_{\omega_{\varepsilon,\delta}},$$
by using the degree condition and the fact that the form is holomorphic.
We remark that the form is $L^2$-integrable by the above discussion and the fact that $u$ has globally bounded coefficients on $X$ (hence on $X \setminus \tilde{Z}_{\varepsilon}$).

Let us observe that by \cite{Dem82} Lemma 3.2,  $(p,0)$-forms get larger $L^2$ norms as the metric increases.
In other words, in bidegree $(p,0)$, the space $L^2(\omega)$  has  the  weakest  topology  of  all  spaces $L^2(\omega_{\varepsilon,\delta})$. 
Indeed,  an  easy  calculation  made  in the above lemma yields
$$|f|^2_{\wedge^{p,0} \omega \otimes h}dV_{\omega} \leq |f|^2_{\wedge^{p,0}\omega_{\varepsilon, \delta}\otimes h}dV_{\omega_{\varepsilon,\delta}}$$
if $f$ is of type $(p,0)$.
By Lebesgue's dominated convergence theorem, we have
$$0 \geq \int_{X\setminus Z_{\varepsilon} }\langle [i \Theta_{L,h},\Lambda]u,u \rangle e^{\Phi_{\varepsilon}} dV_{\omega_{}}$$
by taking $\delta \to 0$.

The rest part of the proof follows in general the proof of \cite{Mou}.
 
Let $-\varepsilon \leq \lambda_1^{\varepsilon} \leq \cdots \leq \lambda_n^{\varepsilon}$ the eigenvalues of $i\Theta_{L,h_{\varepsilon}}$ with respect to $\omega$ on $X \setminus \tilde{Z}_{\varepsilon}$.

Then we have
$$\int_{X \setminus \tilde{Z}_{\varepsilon}} (\lambda_n^{\varepsilon}+\varepsilon) dV_{\omega_{}} \leq \int_{X \setminus \tilde{Z}_{\varepsilon}} ((\lambda_1^{\varepsilon}+\varepsilon)+\cdots+(\lambda_n^{\varepsilon}+\varepsilon)) dV_{\omega_{}}$$
$$\leq \int_{X \setminus \tilde{Z}_{\varepsilon}} (i\Theta_{L,h_{\varepsilon}}+\varepsilon \omega_{}) \wedge \frac{\omega_{}^{n-1}}{(n-1)!}$$
$$\leq \int_{X } (i\Theta_{L,h_{\varepsilon}}+\varepsilon \omega_{}) \wedge \frac{\omega_{}^{n-1}}{(n-1)!}= \int_{X } (c_1(L)+\varepsilon \omega_{}) \wedge \frac{\omega_{}^{n-1}}{(n-1)!}$$
$$\leq \int_{X } (c_1(L)+ \omega_{}) \wedge \frac{\omega_{}^{n-1}}{(n-1)!}=:A.$$
Let $\delta >0$ such that
$$\nu := \frac{n-l}{n-l+1}+ \delta \frac{l-1}{n-l+1} <1.$$
Hence $V_{\varepsilon}:= \{x \in X \setminus \tilde{Z}_{\varepsilon}| \lambda_n^{\varepsilon} +\varepsilon \geq A \varepsilon^{-\delta}\}$ has volume smaller that $\varepsilon^{\delta} \int_X \omega^n$.

On the other hand, by the Monge-Amp\`ere equation, on $X \setminus \tilde{Z}_{\varepsilon}$ we have
$$\prod_{i=1}^n (\lambda_i^{\varepsilon}+\varepsilon)\geq \frac{c}{3}\varepsilon^{n-l}.$$ 
Hence on $X \setminus (V_{\varepsilon} \cup \tilde{Z}_{\varepsilon})$ we have
\begin{align*}
   \lambda_{n-l+1}^{\varepsilon}+\varepsilon & \geq  ((\lambda_{n-l+1}^{\varepsilon}+\varepsilon)\cdots (\lambda_{1}^{\varepsilon}+\varepsilon))^{\frac{1}{n-l+1}} \\
   & \geq c \varepsilon^{\frac{n-l}{n-l+1}} (\lambda_{n}^{\varepsilon}+\varepsilon)^{\frac{l-1}{n-l+1}} \\
      & \geq c \varepsilon^{\nu}.
\end{align*}
Combining this with the Bochner-Kodaira-Nakano inequality, we find
$$0 \geq \int_{X\setminus \tilde{Z}_{\varepsilon}} (\lambda_{1}^{\varepsilon}+\cdots +\lambda_{n-l+1}^{\varepsilon}+\cdots+\lambda_{n-p}^{\varepsilon})|u|^2_{L^{-1},h^{-1}} e^{\Phi_{\varepsilon}} dV_{\omega}$$
$$\geq \int_{X \setminus (\tilde{Z}_{\varepsilon} \cup V_{\varepsilon})}(c\varepsilon^{\nu}-(n-p)\varepsilon)|u|^2_{L^{-1},h^{-1}} e^{\Phi_{\varepsilon}}dV_{\omega}+\int_{ V_{\varepsilon}}-(n-p)\varepsilon |u|^2_{L^{-1},h^{-1}} e^{\Phi_{\varepsilon}}dV_{\omega}.$$
In other words,
$$\int_{X\setminus \tilde{Z}_{\varepsilon}} |u|^2_{L^{-1},h^{-1}} e^{\Phi_{\varepsilon}}dV_{\omega} \leq (1+\frac{n-p}{c\varepsilon^{\nu}-(n-p)}) \int_{ V_{\varepsilon}} |u|^2_{L^{-1},h^{-1}} e^{\Phi_{\varepsilon}}dV_{\omega}$$
$$\leq C \int_{V_{\varepsilon}} \omega^n \leq C \varepsilon^{\nu},$$
where we use that $\Phi_{\varepsilon}$ is uniformly bounded from above.
Since $\tilde{Z}_{\varepsilon}$ is of Lebesgue measure 0,
$$\int_{X\setminus \tilde{Z}_{\varepsilon}} |u|^2_{L^{-1},h^{-1}} e^{\Phi_{\varepsilon}}dV_{\omega}=\int_{X} |u|^2_{L^{-1},h^{-1}} e^{\Phi_{\varepsilon}}dV_{\omega}.$$
Again by Lebesgue's dominated convergence theorem (there is an upper bound by constant), we have
$$\int_{X} |u|^2_{L^{-1},h^{-1}} e^{\Phi}dV_{\omega} \leq 0$$
by taking $\varepsilon \to 0$. This implies that $u=0$ and finishes the proof of the Bogomolov vanishing theorem.
\begin{myrem}
{\rm
In example 1.7 of \cite{DPS94}, we consider a nef line bundle $\cO(1)$ over the projectivisation of the unique non-trivial rank 2 vector bundle as extension of two trivial line bundle over an elliptic curve.
An explicit calculation shows that there exists a unique singular positive metric on $\cO(1)$ whose curvature is the current associated to a smooth curve.
Hence in this example $e(\cO(1))=0$.
But the numerical dimension is $\nd(\cO(1))=1$ since the line bundle is non trivial and not big.
In fact, $(\cO(1))^2=0$.
}
\end{myrem}
\begin{myrem}
{\rm
Our Bogomolov vanishing theorem can be reformulated as follows:

The sheaf of holomorphic $p$-forms over $X$ has no subsheaf of rank one associated to a psef line bundle of numerical dimension strictly larger than $p$.
}
\end{myrem}
According to the fundamental work of Campana \cite{C04} \cite{C11} on special manifolds, the above results suggest to give the following variant of Campana's definition.
\begin{mydef}
Let $L \subset \Omega^p_X$ be a saturated,  coherent  and  rank  one subsheaf. We call it a ``numerical Bogomolov sheaf'’ of $X$ if $\nd(X,L) =p >0$.

We say that $X$ is ``numerically special'' if it has no Bogomolov sheaf.  A compact complex  analytic  space  is  said  to  be  ``numerically special''  if  some  (or  any)  of  its resolutions is ``numerically special''.
\end{mydef}
\begin{myrem}
{\rm
It is conjectured by Campana that specialness is equivalent to the numerical specialness defined here. 

One possibility to address Campana's conjecture would be study the following statement of the Bogomolov vanishing theorem incorporating the numerical dimension instead of the Kodaira-Iitaka dimension:\medskip

{\it For a numerical Bogomolov subsheaf, does there exist a fibration $f:X \to Y$ such  that $L=f^*(K_Y)$ over  the  generic  point  of $Y$ (i.e., $L$ and $f^*(K_Y)$ have the same saturation in $\Omega^p_X$)~?}\medskip

In case the Kodaira dimension case is used, the existence of the fibration comes directly from the Kodaira-Iitaka morphism. However, in case one uses the numerical dimension instead, the existence of the fibration is not guaranteed, i.e.\ there are examples of non abundant numerical Bogomolov sheaves. One can take for instead $X$ to be a Hilbert modular surface obtained as a smooth quotient ${\mathbb D}\times {\mathbb D}/\Gamma$ with an irreducible subgroup $\Gamma\subset{\rm Aut}({\mathbb D})\times{\rm Aut}({\mathbb D})$ (in such a way that no subgroup of finite index of $\Gamma$ splits). It is equipped with two natural foliations ${\mathcal F}$, ${\mathcal G}$ coming from the two factors ${\mathbb D}$, and
$T_X={\mathcal F}\oplus{\mathcal G}$. Then one can check that
${\mathcal F}^*,{\mathcal G}^*\subset\Omega^1_X$ satisfy
$\nd({\mathcal F}^*)=\nd({\mathcal G}^*)=1$, but
$\kappa({\mathcal F}^*)=\kappa({\mathcal G}^*)=-\infty$ (see e.g. \cite{Br03}).
}
\end{myrem}
\textbf{Acknowledgment} I thank Jean-Pierre Demailly, my PhD supervisor, for his guidance, patience and generosity. 
I would like to thank Andreas Höring for some very useful suggestions on the previous draft of this work.
I would also like to express my gratitude to colleagues of Institut Fourier for all the interesting discussions we had. This work is supported by the PhD program AMX of \'Ecole Polytechnique and Ministère de l'Enseignement Supérieur et de la Recherche et de l’Innovation, and the European Research Council grant ALKAGE number 670846 managed by J.-P. Demailly.
  

\begin{thebibliography}{9}
\bibitem[BDPP13]{BDPP}
S. Boucksom, J.P. Demailly, M. Paun and Th.Peternell, {\em The pseudo-effective cone of a compact K\"ahler manifold and varieties of negative Kodaira dimension,} arXiv:math.AG/0405285; to appear in J. Alg. Geometry in 2013.
\bibitem[BEGZ10]{BEGZ}
S\'ebastien Boucksom, Philippe Eyssidieux, Vincent Guedj, Ahmed Zeriahi, {\em Monge-Amp\`ere equations in big
cohomology classes,} Acta Math. 205 (2010), no. 2, 199–262.
\bibitem[Bog78]{Bog}
Fiodor Alekseïevitch Bogomolov, 
 {\em Unstable vector bundles on curves and surfaces,} Proc. Internat. Congress of Math.Helsinki(1978), 517–524.
\bibitem[Bon93]{Bon93}
Laurent Bonavero, 
{\em Inégalités de Morse holomorphes singulières,} C. R. Acad. Sci. Série I
317 (1993), 1163–1166.
\bibitem[Bou02]{Bou02}
S\'ebastien Boucksom,
{\em C\^ones positifs des vari\'et\'es complexes compactes} (French),
Ph. D. Thesis (2002).
\bibitem[Bou04]{Bou04}
S\'ebastien Boucksom, \emph{Divisorial Zariski decompositions on compact complex manifolds}, Ann. Sci. \'{E}cole Norm. Sup. (4) \textbf{37} (2004), no. 1, 45-76.
\bibitem[Br03]{Br03}
Marco Brunella,
{\em Foliations on complex projective surfaces. }(English) Zbl 1070.32502
Dynamical systems. Part II. Topological, geometrical and ergodic
properties of dynamics. Selected papers from the Research Trimester
held in Pisa, Italy, February 4–April 26, 2002. Pisa: Scuola Normale
Superiore. Pubblicazioni del Centro di Ricerca Matematica Ennio de
Giorgi.
arXiv: math.CV/0212082
\bibitem[Cam04]{C04}
 Fréderic  Campana.{\em Orbifolds, special varieties and classification theory.} Ann. Inst.Fourier 54 (2004), 499-665.
\bibitem[Cam11]{C11}
 Fréderic  Campana.  
 {\em Orbifoldes  g\'eom\'etriques  sp\'eciales  et  classification  bim\'ero-morphe des vari\'et\'es K\"ahl\'eriennes compactes.} JIMJ 10 (2011), 809-934.
\bibitem[Dem82]{Dem82}
Jean-Pierre  Demailly.
{\em Estimations $L^2$ pour  l'op\'erateur $\dbar$ d'un  fibr\'e  vectoriel holomorphe semi-positif au-dessus d'une  vari\'et\'e  k\"ahl\'erienne  compl\`ete.}   Ann.   Sci. \'Ecole  Norm.   Sup.   (4),  15(3):457–511, 1982.
\bibitem[Dem85a]{Dem85a}
Jean-Pierre  Demailly.
{\em 
Une preuve simple de la conjecture de Grauert-Riemenschneider, }
Sém. P. Lelong - P. Dolbeault - H. Skoda (Analyse) 1985/86, Lecture Notes in Math. 1295, Springer-Verlag, 24-47.
\bibitem[Dem85b]{Dem85b}
Jean-Pierre  Demailly.
{\em 
Champs magnétiques et inégalités de Morse pour la $\dbar$-cohomologie,} C. R. Acad. Sci. Paris Sér. I Math. 301, 13 mai 1985, 119-122 et Ann. Inst. Fourier (Grenoble) 35 (1985) 189-229.
\bibitem[Dem14]{Dem14}
Jean-Pierre Demailly,
{\em 
On the cohomology of pseudoeffective line bundles,} written exposition of a talk given at the Abel Symposium, Trondheim, July 2013; manuscript Institut Fourier, January 9, 2014

\bibitem[DPS94]{DPS94}
Jean-Pierre Demailly, Thomas Peternell, and Michael Schneider.
{\em Compact complex manifolds withnumerically effective tangent bundles. } J. Algebraic Geom., 3(2):295–345, 1994.
\bibitem[DPS01]{DPS01}
Jean-Pierre Demailly, Thomas Peternell, and Michael Schneider. \emph{Pseudo-effective line bundles on compact Kähler manifolds}.  {International Journal of Mathematics}, 2001, 12(06): 689-741.
\bibitem[Hir64]{Hir64}
Heisuke Hironaka. {\em Resolution of singularities of an algebraic variety over a field of characteristiczero.  I, II.} Ann.  of Math.  (2) 79 (1964), 109–203; ibid.  (2), 79:205–326, 1964.
\bibitem[Mou98]{Mou}
Mourougane, Christophe. {\em Versions k\"ahl\'eriennes du th\'eor\`eme d'annulation de Bogomolov.} Collectanea Mathematica 49.2-3 (1998): 433-445.
\bibitem[Tos19]{Tos}
Valentino Tosatti,
{\em Orthogonality of divisorial Zariski decompositions for classes with volume zero,}
Tohoku Math. J. 71 (2019), no.1, 1-8 
\bibitem[Yau78]{Yau78}
Shing Tung Yau. {\em On  the  Ricci  curvature  of  a  compact  K\"ahler  manifold  and  the  complex  Monge-Amp\`ere equation.  I.} Comm.  Pure Appl.  Math., 31(3):339–411, 1978.
\end{thebibliography}
\end{document}